\documentclass[a4paper,11pt,final]{article}

\usepackage[utf8]{inputenc} 
\usepackage[T1]{fontenc}
\usepackage{lmodern} 
\usepackage[english,french]{babel}
\usepackage{libertine}

\usepackage[a4paper,pdftex,dvips,left=2.2cm, right=2.2cm, top=3.5cm, bottom=3.5cm]{geometry}

\usepackage{amsmath,amsthm,amssymb,amsfonts,mathtools,mathrsfs,amscd}
\usepackage{newtxmath}
\usepackage{xfrac}
\usepackage{faktor}

\usepackage{tikz,pgfplots,adjustbox,graphics,epsfig,svg}

\usepackage[backref=page,colorlinks=false]{hyperref}
\usepackage[capitalize]{cleveref}
\usepackage{cite}

\usepackage{enumerate,float,multicol,footmisc,ragged2e,array,booktabs}
\usepackage[normalem]{ulem}
\usepackage[toc,page]{appendix}
\usepackage{etoolbox}
\input{xy} \xyoption{all}

\def\R{\mathbb R} \def\N{\mathbb N} \def\Z{\mathbb Z}  
  \def\H{\mathbb H}  
   
\def\x{\underline{x}}
\def \Scl{\mathsf{scl}}
\def \scl{\mathit{scl}}
\def \loc{\mathrm{loc}}

\def\dim{\mathsf{dim}}
\def\qand{\quad \text{ and } \quad }

\theoremstyle{plain} 
\newtheorem{lemma}{Lemma}[section]       
\newtheorem{prop}{Proposition}[section]  
\newtheorem{quest}{Question}[section]    

\newtheorem{theo}{Theorem}[section]      

\newtheorem{coro}{Corollary}[section]    

\theoremstyle{definition} 
\newtheorem{definition}{Definition}[section] 

\theoremstyle{remark} 
\newtheorem{remark}{Remark}[section]         

\crefname{thm}{Theorem}{Theorems}
\crefname{lemma}{Lemma}{Lemmas}
\crefname{prop}{Proposition}{Propositions}
\crefname{conj}{Conjecture}{Conjectures}
\crefname{fact}{Fact}{Facts}
\crefname{quest}{Question}{Questions}
\crefname{definition}{Definition}{Definitions}
\crefname{remark}{Remark}{Remarks}
\crefname{example}{Example}{Examples}
\crefname{theo}{Theorem}{Theorems}       
\crefname{coro}{Corollary}{Corollaries}  

\renewcommand*{\backref}[1]{}
\renewcommand*{\backrefalt}[4]{\quad \tiny
  \ifcase #1 (NOT CITED.)%
  \or (Cited on page~#2.)%
  \else (Cited on pages~#2.)%
  \fi
}

\begin{document}
\selectlanguage{english} 
\title{Sets with arbitrary Hausdorff and packing scales in infinite dimensional Banach spaces}
\author{Mathieu Helfter \thanks{Institute of Science and Technology Austria (ISTA), Klosterneuburg, Lower Austria.}}
\date{\today}
\maketitle
\begin{abstract}
For every couple of Hausdorff functions $ \psi$ and $\varphi $ verifying some mild assumptions, there exists a compact subset $ K $  of the Baire space such that the $ \varphi$-Hausdorff measure and the $ \psi$-packing measure on $ K$ are both finite and positive. Such examples are then embedded in any infinite dimensional Banach space to answer positively a question of Fan on the existence of metric spaces with arbitrary scales. 
\end{abstract}
\tableofcontents
\section{Introduction}  
\indent \indent A main purpose of dimension theory is to study the geometric properties of metric spaces using tools from geometric measure theory, such as outer measures and non-integer-valued dimensions. The existence of metric spaces with arbitrary Hausdorff and packing dimensions is a well-established topic. Notably, it follows from a result of Spear \cite{spear1998sets} that there exist subsets of Euclidean spaces with arbitrary values for  Hausdorff and packing dimension as long as the Hausdorff dimension is at most the packing dimension. This article proposes to explore the infinite-dimensional counterparts  of this result. In \cite{helfter2025scales}, the notion of \emph{scale} was introduced to unify and generalize several dimension-like quantities, with the aim of refining the geometric study of infinite-dimensional and $0$-dimensional spaces. The main goal of this article is to answer a question posed by Fan regarding the existence of spaces with arbitrarily prescribed scales. We achieve this by proving in \cref{main1} the existence of Cantor sets of arbitrarily large size whose natural probability measures display a wide variety of scaling behaviors. More precisely, given two functions $\varphi$ and $\psi$ satisfying some mild assumptions, see \cref{cond preceq}, there exists a product of finite sets equipped with an ultrametric distance such that its corresponding $\varphi$-Hausdorff and $\psi$-packing measures are both simultaneously positive and finite. These measures turn out to be proportional to the \emph{equilibrium state}, i.e. the product of equidistributed measures on these finite sets. In \cref{main scl}, we embed these examples into arbitrary infinite-dimensional Banach spaces, thereby demonstrating the existence of compact sets with arbitrary Hausdorff and packing \emph{scales} in any Banach space.

The scaling functions are considered among the class of continuous \emph{Hausdorff functions}: 
\begin{definition}[Hausdorff functions]
The set $  \mathbb{H} $  of  \emph{Hausdorff functions} is the set of continuous non-decreasing functions $ \phi : \R_+ \to \R_+$ such that $ \phi ( 0 )  =  0$ and $ \phi > 0 $ on $ \R_+^*$. 
\end{definition}
 
Hausdorff functions, first introduced by \cite{rogers1998hausdorff}, allow for a generalization of Hausdorff and packing measures by replacing the usual power-law scaling functions $\varepsilon \mapsto \varepsilon^\alpha $ with more general gauge functions.  
Precisely, given $  \phi \in \H$, let us first recall the definition of $\phi$-Hausdorff measure and $ \phi$-packing measure. Let $ (X,d)$ be a (separable) metric space. For the Hausdorff measure, consider an error $ \varepsilon >0$. We recall that an $ \varepsilon$-cover is a countable collection of open balls  $ ( B_i)_{ i \in I } $ of radii at most $ \varepsilon$  so that $ X \subset \bigcup_{ i \in I } B_i$. We then consider the quantity: 
\[ \mathcal{H}_\varepsilon^\phi(X)   := \inf \left\{ \sum_{i \in I } \phi( \vert B_i \vert)   : ( B_i)_{ i \in I } \text{ is an $ \varepsilon$-cover of $X$ }   \right\} \; ,  \] 
where  $ \vert B \vert$ is the radius of a ball $B \subset X$.  
The following non-decreasing limit does exist:
\[ \mathcal{H}  ^\phi( X) := \lim_{ \varepsilon \rightarrow 0} \mathcal{H}_\varepsilon^\phi ( X)  \; . \]
When replacing $(X,d)$ in the previous definitions by any subset of $X$ endowed with the same metric $d$, it is well-known that $ \mathcal{H}^\phi$ defines an outer measure on $X$. This outer measure is usually called the \emph{$ \phi$-Hausdorff measure} on $X$. This construction and well-known properties of Hausdorff measures can be found in \cite{falconer1997techniques,falconer2004fractal,rogers1998hausdorff}. 
Then the packing measure was presented by Tricot \cite{tricot1982two}.  Consider, similarly, an error $ \varepsilon > 0$, and recall that an \emph{$\varepsilon$-pack} of  $ (X,d)$ is a countable collection of disjoint open balls of $X$ with radii at most $ \varepsilon$. Then, set: 
\[ \mathcal{P}_\varepsilon ^{\phi}(X) := \sup \left\{ \sum_{i \in I } \phi (  \vert B_i \vert  ) : (B_i )_ { i \in I  }\ \text{ is an $ \varepsilon$-pack of $X$}\right\} \;  . \]
Since $ \mathcal{P}_\varepsilon^\phi(X)$ is non-increasing as $\varepsilon$ decreases to $0$,  the following quantity is well defined: 
\[ \mathcal{P}_0^\phi (X) := \lim_{ \varepsilon  \rightarrow 0} \mathcal{P}_\varepsilon^\phi(X).\]
The above only defines a pre-measure. We recall that the \emph{$\phi$-packing measure} is given by: 
\[ \mathcal{P}^\phi (X) = \inf  \left\{ \sum_{n \ge  1 } \mathcal{P}_0^{\phi} ( E_n ) : X =  \bigcup_{n \ge 1 } E_n  \right\} \; .  \]
This similarly induces an outer measure on $(X,d)$. 
Note that the initial construction of Tricot in \cite{tricot1982two} considered diameters of balls instead of their radius. See also \cite{evans2018measure,taylor1985packing}. 
 Cutler \cite{cutler1995density} and Haase \cite{haase1986non} indicated that the radius-based definition is doing a better job at preserving desired properties of packing measure and dimension from the Euclidean case. This choice was then followed, for instance, in \cite{mattila1997measure,mcclure1994fractal}. 
\newline
 
The main motivation of this article is to answer \cref{question Fan} of Fan. This question lies in the framework of \emph{scales} that was introduced in \cite{helfter2025scales} to generalize part of dimension theory to infinite (and $0$) dimensional spaces by defining finite invariants that take into account at which "scale" the space must be studied. The involved notions and the answer to that question are  given in \cref{lieu scl}. 
When focusing on packing and Hausdorff measure, Fan's question can be reformulated as: 
\begin{quest} \label{question+}
Given two Hausdorff functions $ \varphi,  \psi \in \H $, under what conditions on $  \varphi $ and $ \psi  $ does there exist a compact metric space $(X,d)$ so that $ \mathcal{H}^\varphi ( X) $ and $ \mathcal{P}^\psi ( X) $  are both finite, non-trivial constants? 
\end{quest}
In that direction, let us mention that De Reyna has shown in \cite{de1988hausdorff} that for any Banach space $ A $  of infinite dimension and any $ \varphi \in \H$  there exists a measurable set $ K \subset A $ such that $ 0 <  \mathcal{H}^\varphi( K) < + \infty$. 

We answer to \cref{question+} in \cref{main1} stated in the coming \cref{lieu statements}. In that same section, we also provide the setting and proper formulation of the question of Fan and provide its answer in \cref{main scl} by embedding examples from \cref{main1} in an arbitrary infinite-dimensional Banach space.
In \cref{lieu proof Banach} we prove \cref{main scl} while \cref{lieu 1} consists of the proof of \cref{main1}. 
 \newline

A last notion that will be involved in the statements of the results is the one of \emph{densities of measure}.  They will allow us to compute Hausdorff and packing measures in our examples.  
\begin{definition} \label{loc rate}
Let  $ \mu $ be a  Borel measure on $(X,d) $. 
Let $ \phi \in \H $ be a Hausdorff function; the \emph{lower and upper $ \phi$-densities} of $ \mu$ are given at $ x \in X $ by: 
\begin{equation} \label{def loc}
\underline{D}^\phi_\mu (x)  = \liminf_{\varepsilon \to 0 }  \frac{\mu ( B (x, \varepsilon) )  }{ \phi ( \varepsilon) }  \qand  \overline{D}^\phi_\mu (x)  = \limsup_{\varepsilon \to 0 }  \frac{\mu ( B (x, \varepsilon) )  }{ \phi ( \varepsilon) } \; .
\end{equation}
\end{definition}  
\vspace{.3cm}
\textit{I am deeply thankful to Ai Hua Fan for posing this insightful question and for his keen interest, and to the anonymous referees for their thoughtful comments.}
\section{Statements \label{lieu statements}}
\subsection{Cantor sets with prescribed Hausdorff and packing measures}
Examples of Cantor constructed here are compact subsets of the space $ E = \N^{ * \N^*} $  of positive integer valued sequences. We endow $ E$  with the ultrametric distance: 
\begin{align*}
\delta \colon E \times E  & \longrightarrow \R^+ \\
(\underline{x}, \underline{x}')  &\longmapsto   2^{  -  \chi ( \underline{x} , \underline{x}') } \; 
\end{align*}
where  $ \chi ( \underline{x} , \underline{x}')  = \inf \lbrace n \ge 1 : x_n \neq x_n' \rbrace$ is the minimal index such that the sequences $ \underline{x}  = ( x_n)_{ n \ge 1} $ and $ \underline{x}' = (x'_n)_{ n \ge 1 } $ differ. 
Note that $ \delta$ provides the product topology on $E$ which is separable. 
We consider compact subsets of $ E$  of the following form: 
\begin{definition}[Compact product and equilibrium state]
A compact subset $ K  \subset E$ is called \emph{compact product} if it is of the form: 
\begin{equation} \label{def X}
 K  = \prod_{ k \ge 1 }   \lbrace 1 , \dots , n_k \rbrace  \quad   \text{ where }  n_k \in \N^*, \  \forall k \ge 1  \; . 
\end{equation} 
It is naturally endowed with the measure $ \mu := \otimes_{ k \ge 1}  \mu _k $ where $ \mu_k $ is the equidistributed probability measure on $ \lbrace 1 , \dots , n_k \rbrace $.  We call this measure the \emph{equilibrium state} $\mu$ of $ K$. 
\end{definition} 
  Note that $ K$ naturally enjoys a group structure as it can be identified to  $ \prod_{ k \ge 1 } \Z \slash n_k \Z$ with the induced product law given by $ \underline{x} + \underline{x}' = ( x_n + x'_n)_{n \ge 1}  $. 
Then, under this consideration, note that $ K $ is a compact topological group and $ \mu$ is the Haar measure on $ K$.   In particular, the $\mu$-mass of a ball centered in $  K $  only depends on its radius. Consequently, for every $ \phi \in \H$, the lower and upper $ \phi$-densities of $ \mu$ are constants on $ K$. We will then identify $\overline{D}_\mu^\phi   $ and $ \underline{D}_\mu^\phi $  with their corresponding constants.
We are now ready to state the first result of this paper:  
\begin{theo} \label{main1}
Let $ \varphi , \psi \in \H $. Assume that there exists a constant $ C > 0 $  such that for every $ \varepsilon > 0 $: 
\begin{equation} \label{cond preceq}
\psi (   2 \varepsilon) \le C \cdot   \varphi (\varepsilon )    \;  . 
\end{equation}
 Then there exists a compact product $ K \subset E $ with equilibrium state $ \mu$  such that: 
\begin{equation*} \label{main2 eq}
\overline{D}_\mu^\varphi \cdot  \mathcal{H}^\varphi ( X)  = \mu (X) = \underline{D}_\mu^\psi  \cdot \mathcal{P}^\psi   ( X) \; . 
\end{equation*}
for every Borel subset $ X \subset K$.  \newline
In particular: 
\begin{equation*}
 \overline{D}_\mu^\varphi \cdot \mathcal{H}^\varphi ( K )   = 1 =  \underline{D}_\mu^\psi  \cdot  \mathcal{P}^{\psi}  ( K ) \; .
\end{equation*}  
\end{theo}
Let us make a few comments about this result. 
\begin{remark}
If $ \psi ( \varepsilon) = \varphi( \varepsilon/2)$ with $ \varphi  \in \H $, then \cref{cond preceq} is verified. 
\end{remark}
\begin{remark}
In the dimensional case: given $ \alpha > 0$, the map $ \varphi( \varepsilon) = \psi ( \varepsilon) =  \varepsilon^\alpha  $   verifies \cref{cond preceq} with  $C = 2^{\alpha}$. 
\end{remark}
\begin{remark}  
The condition in \cref{cond preceq} can be weakened by changing the metric $ \delta$.  For instance, one could check along the proofs that we can use the metric $ \delta ( \underline{x} , \underline{x}') := \rho^{- \chi ( ( \underline{x} , \underline{x}') ) } $ for some constant $ \rho >1 $ and then \cref{cond preceq} becomes $ \psi ( \rho \cdot \varepsilon ) \le C \varphi ( \varepsilon)$ while the conclusion of  \cref{main1} remains the same. 
  
\end{remark}
\begin{remark}\label{rem box}
Also in the setting of \cref{main1}, the upper and lower densities can be defined without considering explicitly the equilibrium state $ \mu$. Indeed, first observe that for every integer $ k \ge 1 $ for every $ \varepsilon \in ( 2^{-(k+1)} , 2^{-k} ]$  and every $ \x \in K $  the mass of the open ball $ B( \underline{x}, \varepsilon ) $ for the equilibrium state $ \mu$ is characterized by: 
\begin{equation} \label{eq mu cover}
 \mathcal{N}_{\varepsilon} (  K ) = \prod_{j=1}^k n_j = \frac{1}{\mu ( B( \underline{x}, \varepsilon ) )   } \; , 
\end{equation}
where $\mathcal{N}_{\varepsilon} (  K )  $ is the $ \varepsilon$-cover number of $ K$, that is the minimal cardinality of a cover of $ K $  by open balls of radius $\varepsilon$. 
This directly implies: 
\begin{equation} \label{dens box}
 \overline{D}_\mu^\varphi \cdot \liminf_{ \varepsilon \to 0 } \mathcal{N}_\varepsilon (K) \cdot \varphi ( \varepsilon) = 1 = \underline{D}_\mu^\psi \cdot   \limsup_{ \varepsilon \to 0 } \mathcal{N}_\varepsilon (K)  \cdot \psi ( \varepsilon) 
\end{equation}
\end{remark} 
The proof of \cref{main1} relies on two ingredients. The first ingredient is \cref{lem geom} which ensures that it is enough to compute the upper and lower densities of $ \mu$ to conclude. The second ingredient, \cref{prop seq main}, constructs explicitly the subset $ K $ by producing the sequence $ v_k := \prod_{j = 1}^k n_j $.  \medskip
\subsection{Compact subspaces with arbitrary scales \label{lieu scl} }
Consider $ (X,d) $ a metric space and $ \mu$ a Borel measure on $ X$. 
Before stating \cref{question Fan} and its answer \cref{main scl}, we shall recall a few definitions from \cite{helfter2025scales}. 
\begin{definition}[Scaling] \label{def scaling}
A one-parameter family of Hausdorff functions $\Scl :=  ( \scl_\alpha )_{ \alpha > 0 } $  is called a (continuous) \emph{scaling} if for every $ \beta > \alpha > 0  $ there exists $ \lambda > 1 $ so that: 
\begin{equation} \label{eq scl}
\scl_\beta  ( \varepsilon ) = o \left( \scl_\alpha ( \varepsilon^\lambda)  \right) \qand \scl_\beta ( \varepsilon ) = o  \left( ( \scl_\alpha  ( \varepsilon)  )^\lambda  \right) \; 
\end{equation}
as $ \varepsilon $ goes to $ 0 $. 
\end{definition}
\begin{definition}[Hausdorff and packing scales] 
Given a scaling $ \Scl = ( \scl_\alpha)_{ \alpha>0}$, the \emph{Hausdorff and packing scales} of $ (X,d)$ are defined by: \[  \Scl_H X := \sup \left\{ \alpha > 0  : \mathcal{H}^{\scl_\alpha}(X) = + \infty  \right\} = \inf \left\{ \alpha > 0  : \mathcal{H}^{\scl_\alpha}(X) = 0   \right\} \; \]
and
\[  \Scl_P X := \sup \left\{ \alpha > 0  : \mathcal{P}^{\scl_\alpha} (X) = +  \infty \right\} = \inf  \left\{ \alpha > 0 : \mathcal{P}^{\scl_\alpha} (X) = 0 \right\}   \; . \]
\end{definition}
It is shown in \cite{helfter2025scales} that the above quantities are always well defined in $ [ 0 , + \infty ] $  and that they verify:
\begin{equation} \label{ineg hp scl}
 \Scl_H X \le \Scl_P X \; , 
\end{equation}
  
extending well-known results from the dimensional case. 
It is easy to check that when $ \Scl = \dim := ( \varepsilon \mapsto \varepsilon^\alpha) $, the condition in \cref{eq scl}  is indeed verified and that we retrieve the classical notions of Hausdorff and packing dimensions. 
Similarly, the followings generalize the classical notions of local (or pointwise) dimensions of a measure:   
\begin{definition}[Local scales of measures]
Let $ x \in X$, the  \emph{lower and upper local scales} of $ \mu $ at $x$ are defined by: 
 \[ \underline{\Scl}_{\loc} \mu (x) := \sup \left\{ \alpha > 0  :  \overline{D}_\mu^{\scl_\alpha} ( x)  =  0  \right\}   \qand  \overline{\Scl}_{\loc} \mu (x) := \inf \left\{ \alpha > 0  :  \underline{D}_\mu^{\scl_\alpha} ( x)  =  + \infty \right\}   \; . \] 
\end{definition}
The above defined scales are always comparable. Theorem $B$ in \cite{helfter2025scales} provides that: 
\begin{equation} \label{ineg loc scl}
\underline{\Scl}_{\loc} \mu (x) \le \Scl_H X \qand \overline{\Scl}_{\loc} \mu (x) \le \Scl_P X \; . 
\end{equation}
for $ \mu$-almost every $x \in X $.  
In the dimensional case, \cref{ineg loc scl} is well-known, see e.g. \cite{falconer1997techniques,fan1994dimensions,tamashiro1995dimensions}.    
 We shall also mention that generalizations of dimension theory, with slightly different viewpoints for general metric spaces were proposed, for instance by McClure \cite{mcclure1994fractal} or Kloeckner \cite{kloeckner2012generalization}.   \newline

One of the features of scales is that they are bi-Lipschitz invariants, i.e. they remain unchanged under bi-Lipschitz transformations of the space. Also, the definition of scaling is tuned to encompass the following examples for every couple of integers  $ p, q \ge 1$:
\begin{equation} \label{ex  growth}
\phi_\alpha : \varepsilon > 0 \mapsto \frac{1}{\exp ^{\circ p  } ( \alpha  \cdot  \log_+^{ \circ q } ( \varepsilon^{-1 } ) ) } \; , 
\end{equation} 

for $ \alpha > 0$ and   where $ \log_+ :=   \1_{ (1, + \infty)  } \cdot \log $ and $ \1_A $ is the indicator function of a set $A$. Recall also that for a self-map $f$ and an integer $ n \ge 1$, we denote \( f^{\circ n} \)  the \( n \)-th iterate of the map \( f \). More explicitly,
\(
f^{\circ 0} = \mathrm{id}, \  f^{\circ (n+1)} = f \circ f^{\circ n } \; .
\)
Note that with $ p = 1 $ and $ q=1$ we retrieve the family defining dimensions.  
For $ p = 2$ and $q=1$ it induces the order that is used to describe several natural infinite-dimensional spaces such as spaces of differentiable maps, see  \cite{tikhomirov1993varepsilon,mcclure1994fractal,helfter2025scales}; ergodic decomposition of measurable maps on smooth manifolds, see \cite{berger2022analytic, berger2017emergence,berger2021emergence,berger2020complexities,helfter2025scales}; or even geometry of gaussian processes such as standard Brownian motion, see \cite{dereich2003link,helfter2025scales} .   
Also, the family given by $ p = 2 $ and $  q = 2$ could be used to describe some compact spaces of holomorphic maps; see \cite{tikhomirov1993varepsilon,helfter2025scales}. \medskip 

The main motivation of this paper is to answer the following question of Aihua Fan: 
\begin{quest}[Fan] \label{question Fan}
Do there exist spaces with arbitrary scales? 
\end{quest}
When restricting to the aforementioned Hausdorff, packing and local scales; the answer to this question is a direct application of \cref{main1}. We actually propose a  stronger result by showing that examples provided by \cref{main1} can be embedded in an arbitrary infinite-dimensional Banach space: 
\begin{theo} \label{main scl}
Let $ ( A , \| \cdot \|  ) $ be an infinite-dimensional Banach space. 
Then, for every scaling $ \Scl = (\scl_\alpha)_{\alpha > 0 } $ and for every $ \beta \ge \alpha > 0$, there exists a compact subset $  X \subset A   $ such that: 
\[ \Scl_H X = \alpha   \qand \Scl_P X  = \beta  \; . \] 
Moreover, there exists a probability measure $ \nu$ on $X$ such that for $ \nu$-almost every $x \in X$:
\[  \underline{\Scl}_{\loc}  \nu ( x) =  \alpha \qand   \overline{\Scl}_{\loc}  \nu (x)  = \beta \; . \]
\end{theo}
This result's first aim is to provide the simplest tools to embed Cantor sets with various scales into infinite-dimensional Banach spaces. Actually, in view of \cref{main1} and de Reyna's result in \cite{de1988hausdorff} it is natural to ask if the following stronger result holds true: 
\begin{quest}
Can we replace $(E, \delta) $ in \cref{main1} by any infinite-dimensional Banach space  ? 
\end{quest}
   
Note that this work is restricted to the study to Hausdorff, packing, and local scales, although other types of scales such as box counting and quantization scales were introduced in~\cite{helfter2025scales}. The question of realizing arbitrary values for these latter scales can be addressed relatively easily using \cref{main scl} and \cref{rem box} together with the fact that both box counting and quantization scales are invariant under topological closure, whereas Hausdorff and packing dimensions are  $\sigma$-stable.  Beyond these, other notions of scales may be introduced such as \emph{conformal},  \emph{Fourier}  or even maybe \emph{Assouad}  scales which generalize their respective dimensional counterparts to appropriate contexts. However, the problem of constructing metric spaces with prescribed values for these more refined scales is subtler. In the finite dimensional setting, some results are already known in this direction. Spear showed in~\cite{spear1998sets} that there exist subsets of the interval $[0,1]$ with arbitrary Hausdorff, packing, and box dimensions, provided they satisfy the inequalities:
\[
0 \leq \dim_H \leq \underline{\dim}_B \leq 1 \quad \text{and} \quad 0 \leq \dim_H \leq \dim_P \leq \overline{\dim}_B \leq 1 \;.
\]
More recently, Ishiki proved in~\cite{ishiki2021fractal} that there exists a Cantor ultrametric space whose Hausdorff, packing, upper box-counting, and Assouad dimensions can be prescribed arbitrarily, subject to the constraints:
\[
\dim_H \leq \dim_P \leq \overline{\dim}_B \leq \dim_A \;.
\]
It is natural to wonder if the  significant flexibility appearing in such examples could be broadcasted to infinite-dimensional settings. 
\section{Invariance of scales and embedding compact products in  Banach spaces }\label{lieu proof Banach}
In this section we study quasi-Lipschitz invariance of scales and then deduce \cref{main scl}  from \cref{main1}. 
\subsection{Quasi-Lipschitz invariance of scales}
It has been proved in \cite{helfter2025scales} that scales are bi-Lipschitz invariants. Actually in \cref{scl biquasi}, we obtain a slightly stronger invariance, which will be a key tool of the proof of \cref{main scl}. This result relies on the following  notions: 
\begin{definition}[Quasi-Lipschitz map and embedding]
Let $(X,d_X) $ and $ (Y, d_Y) $ be two metric spaces.  
We say that  $f : X \to Y $ is a \emph{quasi-Lipschitz map} if it is locally $ \alpha$-H{\"o}lder for every $ \alpha <1$; i.e.
\[ \lim_{ \varepsilon \to 0 }  \  \inf_{ 0 <   d_X (  x,x') <  \varepsilon}  \frac{\log ( d_Y( f(x), f(x') ) ) }{ \log( d_X ( x,x') )  } \ge 1 \; . \] 
Moreover, if $ f $ is injective, it is said  to be a \emph{quasi-Lipschitz embedding} if $f^{-1} : f(X) \subset Y \to X $ is also a quasi-Lipschitz map, or equivalently for every $x \in X$:
\[\lim_{x' \to x, \ x' \neq x}  \frac{\log ( d_Y( f(x), f(x') ) ) }{ \log ( d_X ( x,x') )  } = 1 \; ;\]
moreover the convergence is uniform in $  x \in X$.  
\end{definition} 

Then we state the following result for quasi-Lipschitz maps: 
\begin{lemma} \label{scl quasiL}
Let $ ( X,d_X) $ and $ (Y, d_Y)$ be two metric spaces and $ \Scl$ be a scaling. Let  $f : X \to Y $ be a quasi-Lipschitz map, then: 
\[ \Scl_H f(X) \le \Scl_H X  \qand \Scl_P f(X) \le \Scl_P X  \;. \]
Moreover, for every measure $ \mu$ on $X$, with $ \nu := f_* \mu $, the pushforward by $\mu$ of $f$, every $x \in X $ verifies:
\[ \underline{\Scl}_{\loc} \nu ( f(x)) \le  \underline{\Scl}_{\loc}  \mu (x) \qand  \overline{\Scl}_{\loc} \nu ( f(x)) \le  \overline{\Scl}_{\loc}  \mu (x)  \; . \] 
\end{lemma} 
\begin{proof}
First observe that it is sufficient to prove that for every $ \beta > \alpha > 0$, the following inequalities hold: 
\[ \mathcal{H}^{\scl_\beta}  ( f(X) ) \underset{\textsc{(i)}}{\le}  \mathcal{H}^{\scl_\alpha}( X )  , \quad  \mathcal{P}^{\scl_\beta} (f(X) ) \underset{\textsc{(ii)}}{\le} \mathcal{P}^{\scl_\alpha} (X  ) \] 
and for every $x \in X$:
\[   \overline{D}^{\scl_\alpha}_\mu(x) \underset{\textsc{(iii)}}{\le}  \overline{D}^{\scl_\beta}_\nu (f(x))   , \quad   \underline{D}^{\scl_\alpha}_\mu( x)  \underset{\textsc{(iv)}}{\le}  \underline{D}^{\scl_\beta}_\nu     f(x)  \;.  \] 
Indeed, then we conclude by the definition of the involved scales. 
To prove these inequalities, take $ \beta > \alpha> 0$. By \cref{def scaling} of scaling, there exists $  0 < \kappa < 1 $ and $ \varepsilon_0 > 0$ such that for every $ 0<  \varepsilon < \varepsilon_0$: 
\begin{equation}\label{eq scl bilip}
\scl_\beta ( \varepsilon^{\kappa} )  < \scl_\alpha( \varepsilon ) \; . 
\end{equation}
As $f$ is quasi-Lipschitz, we can also assume that  $ \varepsilon_0 >0$  is small enough so that for every $ x , x' \in X $ so that $ d_X ( x,y) < \varepsilon_0$, it holds: 
\[ \frac{\log ( d_Y( f(x), f(x') ) ) }{ \log ( d_X ( x,x') )  }  >  \kappa \; , \] 
or equivalently: 
\begin{equation} \label{lip eq scl}
d_Y( f(x), f(x') )  <  ( d_X ( x,x') )^{\kappa} \; . 
\end{equation}
\newline \vspace{0.2cm}
\textsc{Proof of the inequality (i) on Hausdorff measures:} \newline
Consider $ 0 <  \varepsilon < \varepsilon_0 $. For every  countable set $J$ and every $ \varepsilon$-cover  $( B (x_j, \varepsilon_j) )_{j \in J} $ of $X$, it holds: 
\[ f(X)  \subset \bigcup_{j \in J } B ( f(x_j) ,    \varepsilon_j^\kappa ) \; . \]
Then $( B( f(x_j)  , \varepsilon_j^\kappa ) )_{ 1 \le j \le N} $ is a $\varepsilon^\kappa$-cover of $ f(X) $.
By \cref{eq scl bilip}, we obtain: 
\[ \mathcal{H}^{\scl_\beta}_{ \varepsilon^\kappa} ( f(X)) \le \sum_{  j \in J } \scl_\beta ( \varepsilon_j^\kappa) \le \sum_{  j \in J } \scl_\alpha ( \varepsilon_j)  \; .   \]
As this holds true for any such cover, we obtain:
\[ \mathcal{H}^{\scl_\beta}_{ \varepsilon^\kappa} ( f(X))  \le  \mathcal{H}^{\scl_\alpha}_{ \varepsilon} ( X)   \; . \] 
Taking the limit as $ \varepsilon$ goes to $ 0$ provides the desired inequality on Hausdorff measures. 
 \newline \vspace{0.2cm} 

\textsc{Proof of the inequality (ii) on packing measures:}\newline
Similarly, consider $ 0 <  \varepsilon < \varepsilon_0^{1/\kappa} $. Let $ E\subset X$. For every  countable set $J$ and every $ \varepsilon^\kappa$-packing  $( B (f(x_j), \varepsilon_j^\kappa) )_{j \in J} $ of $f(E)$, the family  $( B (x_j, \varepsilon_j) )_{j \in J} $ is an $\varepsilon$-pack of $E$. Thus, still by \cref{eq scl bilip}, it follows:  
\[   \sum_{  j \in J } \scl_\beta ( \varepsilon_j^\kappa)  \le \sum_{  j \in J } \scl_\alpha ( \varepsilon_j)  \le  \mathcal{P}_{\varepsilon}^{\scl_\alpha}  (E)  \; .\] 
As this holds for any such $ \varepsilon^\kappa$-pack, taking $\varepsilon$ small provides $\mathcal{P}_0^{\scl_\beta}  (f(E))  \le  \mathcal{P}_0^{\scl_\alpha} (E)$.  
Now as $E$ is an arbitrary subset of $X$, it follows by the definition of packing measure that $\mathcal{P}^{\scl_\beta} (f(X) ) \le \mathcal{P}^{\scl_\alpha}( X )$. 
\newline \vspace{0.2cm} 

\textsc{Proof of the inequalities (iii) and (iv) on densities of measures:}\newline
Still consider $ \varepsilon$ small and observe that for every $x \in X$, the following sequence of inequalities holds:  
\[ \nu ( B ( f(x) , \varepsilon^\kappa ) ) \ge \nu ( f ( B(x,\varepsilon) )) = \mu ( f^{-1} (  f(B(x,\varepsilon)))) \ge \mu ( B(x, \varepsilon) ) \; .  \]
Then it follows:   
\[ \frac{\mu(B(x,\varepsilon))}{\scl_\alpha ( \varepsilon) }  \le \frac{\nu(B(f(x),\varepsilon^\kappa))}{\scl_\alpha ( \varepsilon) } \le \frac{\nu(B(f(x),\varepsilon^\kappa))}{\scl_\beta( \varepsilon^\kappa) } \; . \]
Taking the $ \limsup $ and $ \liminf$ as $ \varepsilon $ goes to $ 0$ provides the desired results. 
\end{proof}
As a direct application of the above \cref{scl quasiL}, we obtain: 
\begin{coro} \label{scl biquasi}
Let $ ( X,d_X) $ and $ (Y, d_Y)$ be two metric spaces and let $ \Scl$ be a scaling. \newline
Let  $f : X \to Y $ be a quasi-Lipschitz embedding, then their Hausdorff and packing scales coincide: 
\[ \Scl_H f(X) = \Scl_H X  \qand \Scl_P f(X) = \Scl_P X  \;. \]
Moreover, for every measure $ \mu$ on $X$ with $ \nu := f_* \mu $, the pushforward of $ \mu$ by $ f $ is such that for every $x \in X $:
\[ \underline{\Scl}_{\loc} \nu ( f(x)) =   \underline{\Scl}_{\loc}  \mu (x) \qand  \overline{\Scl}_{\loc} \nu ( f(x)) =  \overline{\Scl}_{\loc}  \mu (x)  \; . \] 
\end{coro} 
We will apply the above \cref{scl biquasi} in the coming section. 
\subsection{Embeddings in Banach spaces: proof of \cref{main scl}} \label{lieu proof scl}
We now prove \cref{main scl}. A first step is the following result: 
\begin{lemma} \label{prop embed}
Let $ (A , \| \cdot \| )$ be an infinite-dimensional Banach space, then there exists a quasi-Lipschitz embedding $ f : ( E , \delta ) \hookrightarrow ( A , \|\cdot \|)$.
\end{lemma} 
\begin{proof}
Up to replacing the norm $ \| \cdot \| $ by some equivalent norm --  corresponding to a bi-Lipschitz transformation of $A$ -- we can assume by Theorem $1$ in  \cite{mercourakis2014equilateral} that $A$ contains an \emph{ infinite equilateral set}; that is, a countable collection $ (a_n)_{ n \ge 1 } $ of vectors of $ A$ such that for every $ n\neq m$ it holds $ \| a_n - a_m \| = 1 $. 
  Thus we define the embedding: 
  \[ f : \x \in E \hookrightarrow \sum_{ k \ge 1 } \frac{a_{x_k}}{k \cdot 2^k } \in A   \; . \] 
 As $ A$ is a Banach space and the latter sum is normally convergent, the above map $i$ is well defined.  To conclude, it suffices to show that $i$ is a quasi-Lipschitz embedding. To do so, consider $\x, \x' \in  E$ and let $ k_0 := \chi ( \x , \x')   $ be the minimal index so that $ \x $ and $ \x' $ differ. Then we obtain the following inequalities: 
\begin{align*} 
 \| f( \x ) - f (  \x') \| &\ge  k_0^{-1} 2^{-k_0}  - \sum_{k > k_0} k^{-1} 2^{-k}    \\
  &\ge ( k_0^{-1} - (k_0+1)^{-1} )  \cdot 2^{-k_0}  \\
  &\ge   2^{-(k_0 +2)} \cdot k_0^{-2}  \; , 
\end{align*}
that is:  
\begin{equation} \label{eq ilip}
 \| f( \x ) - f (  \x') \| \le \frac{1}{4}  \delta( \x, \x') \left( \log_2 \delta( \x, \x')  \right)^2 \;. 
\end{equation}
Conversely: 
\begin{equation} \label{eq lip}
 \| f( \x ) - f (  \x') \| \le   \sum_{ k \ge k_0 }  k^{-1} \cdot 2^{-k}  \le  2^{- k_0  +1 }  = 2 \delta( \x , \x') \; . 
\end{equation} 
Combining \cref{eq lip} and \cref{eq ilip} provides that for $ \x \neq \x'$, the following sequence of inequalities hold:
\[    1  - \frac{  \log 2  }{  \vert \log \delta( \x , \x') \vert } \le  \frac{\log ( \| f ( \x ) - f (\x') \| ) }{ \log \delta( \x , \x') }  \le  1 + \frac{   \vert \log  ( \log_2 \delta ( \x, \x') )^2  \vert  + \log 4  }{ \vert \log \delta( \x , \x') \vert  }    \; . \]
Taking $ \x'$ arbitrarily close to $ \x $ for every $ \x \in E$ provides that $ f$ is indeed a quasi-Lipschitz embedding. 
\end{proof}
We conclude this section with the proof of \cref{main scl}: 

\begin{proof}[Proof that \cref{main1} implies \cref{main scl}]
Let  $ \varphi :=  \varepsilon \mapsto  \scl_\alpha (\varepsilon)  $ and $ \psi := \varepsilon \mapsto \scl_\beta (\varepsilon /2) $. Obviously by \cref{def scaling} of scaling, the condition given by \cref{cond preceq} is verified. 
Thus by \cref{main1}, there exists a compact product $ K \subset E $ so that $ \mathcal{H}^\varphi ( K ) 
$ and $ \mathcal{P}^\psi ( K ) $ are both finite, non-trivial and proportional to its equilibrium state $ \mu$.
Then it immediately follows that: 
\[ \Scl_H  K = \alpha \qand \Scl_P  K = \beta   \; . \] 
Moreover, the local scales of the equilibrium state $\mu$ of $K$ at a point $\x \in K $ are equal to:
\[  \underline{\Scl}_{\loc} \mu (\x) = \alpha \qand  \overline{\Scl}_{\loc} \mu (\x) = \beta  \; .   \] 
Let then $ f : E \to A $ be the quasi-Lipschitz embedding provided by \cref{prop embed}. 
Then picking $ X := f ( K) $, $ \nu = f_* \mu$ and applying \cref{scl biquasi} allows us to conclude the proof of \cref{main scl}. 
\end{proof}
\section{Hausdorff and packing measure on compact products} \label{lieu 1}
\subsection{Densities of the equilibrium state}
We first link densities of an equilibrium state with Hausdorff and packing by the following: 
\begin{lemma} \label{lem geom}
Let  $ K \subset E  $ be a compact product with equilibrium state $ \mu$. Then for every Borel subset $ X \subset K $ it holds: 
 \[ \mathcal{H}^\phi (X)  = \frac{1 }{\overline{D}_\mu^\phi}  \mu (X)  \quad \text{ if  } \   0 < \overline{D}_\mu^\phi < + \infty  \; \] 
and 
\[ 
 \mathcal{P}^\phi( X)  = \frac{1 }{\underline{D}_\mu^\phi} \mu (X) \quad \text{ if }  \  0  <  \underline{D}_\mu^\phi < + \infty  \; .   \] 
\end{lemma} 
The proof of the equality for packing measure can be, for instance, directly deduced from a more general result of Edgar \cite{edgar2000}[Theorem 2.5] relating packing measures to extrema of the lower densities of a measure verifying a \emph{strong Vitali property}; see \cite{edgar2000}[Section $2$]. 
However, in the considered examples, the proof is quite straightforward so we provide it for the sake of completeness: 
\begin{proof}
We first show: 
\begin{equation} \label{goal dens lem1}
 \mathcal{H}^\phi ( K ) \cdot \overline{D}^\phi_\mu \underset{\textsc{(a)}}{=} 1   \quad \text{ if } \  0 < \overline{D}_\mu^\phi < + \infty   \qand  \mathcal{P}_0^\phi (K)   \cdot \underline{D}_\mu^\phi \underset{\textsc{(b)}}{=} 1 \quad \text{ if } \   0 < \underline{D}_\mu^\phi < + \infty  \; . 
\end{equation}
\textsc{Proof of the equality (a) for Hausdorff measure:}\newline
Fix $ \delta > 0 $. Then for every sufficiently small $ \varepsilon > 0$, every open ball $B$ of radius at most $ \varepsilon$ verifies: 
\begin{equation} \label{eps small}
 \phi ( \vert B \vert ) \ge   \frac{\mu ( B) }{ \overline{D}_\mu^\phi + \delta} \; . 
\end{equation}
For such a small $ \varepsilon$,  consider  an $ \varepsilon$-cover $ ( B_j)_{ j \in J} $  of $ K$. Then by \cref{eps small}, it follows: 
\begin{equation}
\sum_{j \in J  }  \phi ( \vert  B_j \vert  ) \ge  \frac{1}{\overline{D}_\mu^\phi + \delta}    \sum_{j \in J  } \mu ( B_j) \ge   \frac{1}{\overline{D}_\mu^\phi + \delta}   \; .
\end{equation}
As this holds for any $ \varepsilon$-cover and $ \delta$ can be taken arbitrarily small, we obtain: 
\[ \mathcal{H}^\phi ( K) \ge  \frac{1}{\overline{D}_\mu^\phi } \; . \] 
To show the reverse inequality, fix again $ \delta >0 $ and note that for every $ \varepsilon > 0 $ there exists $  \eta =  2^{-k} \in (0 ,  \varepsilon) $ such that a ball of radius $ \eta $ has its mass greater than $ \phi ( \eta) \cdot (\overline{D}_\mu^\phi - \delta)^{-1}$. Then for the minimal cover $( B_j)_{ j \in J} $ of $ K $ by balls of radius $ \eta$, i.e. $ J  $ has cardinal $ \mu( B_j)^{-1}$ for every $ j \in J $, we obtain: 
\[  \mathcal{H}_\varepsilon^\phi ( K  ) \le  \sum_{ j  \in J } \phi ( \eta ) \le \frac{1}{\overline{D}_\mu^\phi - \delta}   \sum_{ j  \in J } \mu ( B_j )  =    \frac{1}{\overline{D}_\mu^\phi - \delta} \; . \] 
Taking $ \varepsilon$ and $\delta$ small provides $ \mathcal{H}^\phi ( K)  \le \frac{1}{\overline{D}_\mu^\phi }$ and allows us to conclude the proof of the equality for the Hausdorff measure.  \newline
\textsc{Proof of the equality (b) for the packing measure:}
The proof for packing measure is actually quite similar.
Fix $ \delta > 0 $. For every sufficiently small $ \varepsilon > 0$, every ball $B$ of radius at most $ \varepsilon$ verifies: 
\begin{equation} \label{eps small2}
 \phi ( \vert B \vert ) \le   \frac{\mu ( B) }{ \underline{D}_\mu^\phi  - \delta} \; . 
\end{equation}
For such a small $ \varepsilon$,  consider  an $ \varepsilon$-packing $ ( B_j)_{ j \in J} $  of $ K$. Then by \cref{eps small2}, it follows: 
\begin{equation}
\sum_{j \in J  }  \phi ( \vert  B_j \vert  ) \le  \frac{1}{\underline{D}_\mu^\phi - \delta}    \sum_{j \in J  } \mu ( B_j) \le   \frac{1}{\overline{D}_\mu^\phi - \delta}   \; .
\end{equation}
As this holds for any $ \varepsilon$-pack and $ \delta$ can be taken arbitrarily small, we obtain: 
\[ \mathcal{P}_0^\phi ( K) \le  \frac{1}{\underline{D}_\mu^\phi } \; . \] 
To show the reverse inequality, fix again $ \delta >0 $ and note that for every $ \varepsilon > 0 $ there exists $  \eta =  2^{-k} \in (0 ,  \varepsilon) $ such that a ball of radius $ \eta $ has its mass smaller  than $ \phi ( \eta) \cdot (\underline{D}_\mu^\phi + \delta)^{-1}$. Then, the minimal cover $( B_j)_{ j \in J} $ of $ K $ by balls of radius $ \eta$  is an $ \varepsilon$-pack and thus verifies: 
\[  \mathcal{P}_\varepsilon^\phi ( K  ) \ge  \sum_{ j  \in J } \phi ( \eta ) \ge \frac{1}{\underline{D}_\mu^\phi + \delta}   \sum_{ j  \in J } \mu ( B_j )  =    \frac{1}{\underline{D}_\mu^\phi  + \delta} \; . \] 
Taking $ \varepsilon$ small provides then $ \mathcal{P}_0^\phi ( K) \ge  \frac{1}{\underline{D}_\mu^\phi }$ which concludes the proof of that second equality. \newline
We  now finish the proof of \cref{lem geom}. 
Let $ B$ be an arbitrary ball of $X$ with radius $ r > 0$. Let $ ( B_j)_{ 1 \le j \le N} $ be the $ N = \frac{1}{\mu (B) }$ disjoint balls of radius $ r$. Then, for $\varepsilon < r$,  any $\varepsilon$-cover (resp. $\varepsilon$-pack) can be partitioned into $\varepsilon$-covers (resp. $\varepsilon$-packs)  of the balls $ (B_j)_{1 \le j \le N} $. Now as all the balls $B_j$ are isometric to $B$ it follows: 
\[  \mathcal{H}^\phi ( K) = \sum_{j=1}^N \mathcal{H}^\phi ( B_j) = N \cdot  \mathcal{H}^\phi ( B) \qand \mathcal{P}_0^\phi ( K) = \sum_{j=1}^N \mathcal{P}_0^\phi ( B_j) = N \cdot  \mathcal{P}_0^\phi ( B)  \; .  \] 
We have just shown: 
\[  \mathcal{H}^\phi ( B) = \mathcal{H}^\phi ( K) \cdot \mu(B) \qand  \mathcal{P}_0^\phi ( B) = \mathcal{P}_0^\phi ( K) \cdot \mu(B)  \; .  \] 
Finally, as $ \mathcal{H}$  and $ \mathcal{P}_0$ are pre-measures on $K$, we obtain the desired equality for every subset $X $ of $ K$ by  Carath{\'e}odory's extension theorem and \cref{goal dens lem1}. 
\end{proof}
  
The latter lemma provides that it is sufficient to evaluate densities of equilibrium states to obtain the corresponding Hausdorff and packing measures. Moreover, these densities are given by the following lemma. 

\begin{lemma}\label{fact contdiscr}   Let $ \phi \in \H $.
The densities of the equilibrium state $ \mu$ are given for every $ \x \in  K$ by: 
\begin{equation}
\underline{D}_\mu^\phi = \liminf_{ k \to + \infty} \frac{\mu ( B (\x,2^{-k} ))  }{ \phi( 2^{-k} ) }   \qand \overline{D}_\mu^\phi = \limsup_{ k \to + \infty} \frac{\mu ( B (\x,2^{-k} ))  }{ \phi( 2^{-(k+1)} ) }  \; ,  
\end{equation}   
where $ B( \underline{x} , \varepsilon) $ is the open ball of radius $ \varepsilon $ centered at $ \underline{x}$. 
\end{lemma}
  
\begin{proof}
Consider $ \x \in K$. 
First note that: 
\[ \liminf_{\varepsilon \to 0} \frac{\mu ( B (\x,\varepsilon))}{\phi(\varepsilon)} \le \liminf_{k \to + \infty} \frac{\mu ( B (\x,2^{-k}))}{\phi(2^{-k})} \; , \]
providing:
\begin{equation} \label{eq cd1}
  \underline{D}_\mu^\phi \le \liminf_{ k \to + \infty} \frac{\mu ( B (\x,2^{-k} ))  }{ \phi( 2^{-k} ) } \; . 
\end{equation} 
Now observe that for every $ \varepsilon \in ( 0,1) $ , there exists a unique integer $ k $ such that $2^{-(k+1)} < \varepsilon \le 2^{-k}$. It verifies $ B ( \x , \varepsilon) = B ( \x , 2^{-k} )$, and thus as $ \phi $ is non-decreasing, we obtain: 
\begin{equation*} 
\frac{\mu(B(\x, 2^{-k} ) ) }{\phi(2^{-k}) } \le  \frac{\mu(B(\x, \varepsilon	 ) ) }{\phi(\varepsilon ) }  \le \frac{\mu(B(\x, 2^{-k} ) ) }{\phi(2^{-(k+1)}) } \; .
\end{equation*}
As such a $k$ exists for every $ \varepsilon < 1$ we obtain:
\begin{equation}\label{eq cd2}
  \liminf_{ k \to + \infty} \frac{\mu ( B (\x,2^{-k} ))  }{ \phi( 2^{-k} ) }   \le \underline{D}_\mu^\phi   \le  \overline{D}_\mu^\phi \le  \limsup_{ k \to + \infty} \frac{\mu ( B (\x,2^{-k} ))  }{ \phi( 2^{-(k+1)} ) }
\end{equation}
   
For every fixed integer $ k \ge 1$, by continuity and positivity of $ \phi $ at $2^{-(k+1)}$, then any $ \varepsilon_k \in [ 2^{-k} , 2^{-k+1} ) $ sufficiently close to $ 2^{-(k+1)} $  verifies: 
 \[  \vert  \phi (  \varepsilon_k) - \phi ( 2^{-(k+1)} ) \vert  \le \frac{\phi( 2^{-(k+1)} )  }{k}  \;.  \] 
Fixing such a value of $ \varepsilon_k$ for every integer $k$ provides a sequence $ ( \varepsilon_k)_{ k \ge 1 } $ such that for every $ k \ge 1$:
\[ 2^{-(k+1)} < \varepsilon_{k} \le 2^{-k} \qand  \lim_{k \to +\infty} \frac{\phi(2^{-(k+1)} ) }{\phi ( \varepsilon_k) } = 1 \; . \; \]

Then by writing: 
\[ \frac{\mu ( B ( \x  , \varepsilon) ) }{\phi ( \varepsilon) } = 
\frac{\mu ( B ( \x  , 2^{-k} ) ) }{\phi(2^{-(k+1)}) } \cdot  \frac{\phi(2^{-(k+1)} ) }{\phi ( \varepsilon) } 
\; , \]
for every $k \ge 1$ and taking the limit as $k$ goes to infinity, we obtain: 
\begin{equation}\label{eq cd3}
\overline{D}_\mu^\phi \ge \limsup_{ k \to + \infty} \frac{\mu ( B ( \x  , \delta_k) ) }{\phi ( \delta_k) } =\limsup_{ k \to + \infty} \frac{\mu ( B (\x,2^{-k} ))  }{ \phi( 2^{-(k+1)} ) }  \;. 
\end{equation}
Combining \cref{eq cd1,eq cd2,eq cd3} concludes the proof of \cref{fact contdiscr}. 
\end{proof}

\subsection{Construction of the compact products and proof of \cref{main1} } 
We now provide the elementary construction of the adapted sequence of cardinals of the corresponding compact product as stated below in \cref{prop seq main}.  We first introduce a few notations. 
Consider the set $ R$ of non-decreasing unbounded positive sequences: 
 \[   R := \left\{ (a_k)_{ k \ge 1} \in  \mathbb{R_+^*}^{\N^*} :  \lim_k a_k = + \infty \qand  a_{k+1} \ge a_k , \ \forall k \ge 1 \right\}  \; . \]
 We also denote $ \underline{a} = (a_k)_{ k \ge 1 } $ as an element of $R$  and use this same notation for $b,u,v$ and $n$. 
We shall write $ \underline{a} \le \underline{b}$ if the sequences  $\underline{a} , \underline{b}  \in R $ verify $ a_k \le b_k $ for every $ k \ge 1$. The second ingredient in the proof of  \cref{main1} is: 
\begin{prop} \label{prop seq main}
Let $ \underline{a} \le \underline{b} $ be two elements of $ R $. Then there exists a sequence of positive integers $ \underline{v} \in E $ such that  $ v_{k} $ divides $ v_{k+1}$ for every $ k \ge 1$, while:
\begin{equation} \label{prop main ineq}
1 \le \limsup_{ k \to + \infty}  \frac{a_k}{v_k}  \le 2  \qand  1 \le \liminf_{ k \to + \infty} \frac{b_k}{v_k} \le 2 \; . 
\end{equation}
\end{prop}
  
This proposition is proven below using the following Lemmas \ref{osc lemma} and \ref{r to prod}. Let us first see how \cref{lem geom}, \cref{prop seq main} and \cref{fact contdiscr} allow us to obtain:    
\begin{proof}[Proof of \cref{main1}]
As Hausdorff and packing measures are linear with respect to Hausdorff functions, we can assume that $ C = 1 $ in \cref{cond preceq}, and this is up to multiplying $ \varphi $ or $ \psi $ by a scalar. \newline
Let then $ \underline{a}, \underline{b} \in R $ be the sequences defined for $ k \ge 1$ by: 
\begin{equation} \label{def ab}
 a_k :=  \frac{1}{\varphi ( 2^{-(k+1)} ) } \qand b_k :=  \frac{1}{\psi ( 2^{-k} ) } \; .
\end{equation}
Thus by \cref{cond preceq}, since $ C= 1$, the inequality $ \underline{a} \le \underline{b}$ obviously holds. Let then $ \underline{v} \in E  $ be the sequence provided by \cref{prop seq main} and consider the sequence $ \underline{n} \in E $ defined for $k \ge 1 $ by $ n_{k} = \frac{v_k}{v_{k-1}} \in \N^*$ with $v_0 := 1$. The compact product that we consider is:
 \begin{equation} \label{def K}
 K := \prod_{ k \ge 1} \lbrace 1 , \dots , n_k \rbrace \subset E \; . 
\end{equation}
 
Then, \cref{fact contdiscr} applied to $ \phi = \varphi $ and then $ \phi = \psi$ provides:  
\begin{equation*} 
\underline{D}_\mu^\psi = \liminf_{ k \to + \infty} \frac{\mu ( B (\x,2^{-k} ))  }{ \psi( 2^{-k} ) }  =  \liminf_{ k \to + \infty} \frac{b_k}{v_k}  \qand \overline{D}_\mu^\varphi =\limsup_{ k \to + \infty} \frac{\mu ( B (\x,2^{-k} ))  }{ \varphi( 2^{-(k+1)} ) }  =  \limsup_{ k \to + \infty} \frac{a_k }{v_k }  \; .   
\end{equation*}
Thus by \cref{prop seq main}, we obtain: 
\[ 1 \le \underline{D}_\mu^\psi \le 2 \qand 1 \le \overline{D}_\mu^\varphi \le 2  \;. \]
As $\underline{D}_\mu^\psi $ and $  \overline{D}_\mu^\varphi$ are both finite and non-zero, we conclude the proof by a direct application of \cref{lem geom}.
\end{proof}

\begin{lemma} \label{osc lemma}
For every $ \underline{a} \le \underline{b} \in R$, there exists  $ \underline{u} \in R $  with $ \underline{a} \le \underline{u} \le \underline{b}$ and there exists an increasing sequence of integers $( T_\ell)_{ \ell \ge 1} $ such that: 
\begin{equation}
u_{T_\ell} = a_{T_\ell} \qand u_{T_{\ell}+1} = b_{T_{\ell}+1}  \; 
\end{equation}
for every $ \ell \ge 1$. 
\end{lemma}
\begin{proof}
Let $ T_0 = 0$. For $ \ell \ge 1 $, we define recursively: 
\begin{equation} \label{def time}
T_{\ell +1 } := \inf \lbrace  k >  T_\ell + 1  : a_{k} > b_{T_\ell+1} \rbrace  \; . 
\end{equation}
As  $ \underline{a} $  grows to infinity, each $ T_\ell $ is finite and well defined.
Moreover, $ ( T_\ell)_{\ell\ge 1} $ is increasing.\newline
Then define the sequence $ \underline{u}$ for $ k \ge 1$ by:
\begin{equation} \label{def vk alt}
   u_{k}:=  
\begin{dcases} 
    \hspace{1cm} a_{T_\ell}  & \text{ if }  k = T_\ell \text{ with } \ell \ge 1   , \\
  \hspace{1cm}   b_{T_\ell +1 }            &  \text{ if } T_\ell <  k < T_{\ell+1} \text{ with } \ell \ge 0  \; . 
\end{dcases}
\end{equation}
It  is then clear by construction that the sequence $\underline{u} $ satisfies the desired properties. 
\end{proof}
Given a sequence $\underline{u} \in R$, we can always find a product of integers growing like $\underline{u}$ according to the following:  
\begin{lemma} \label{r to prod}
For every $\underline{u} \in R$, there exists a sequence of positive integers $\underline{v}= ( v_k)_{k \ge 1 } \in E$ such that for every $ k \ge 1 $ the term $ v_{k} $ divides $ v_{k+1} $ and for $k$ sufficiently large, it holds: 
\begin{equation*}
 \frac{u_k}{2}   \le v_k \le u_k \; . 
\end{equation*}
\end{lemma}
\begin{proof}
Let us define the sequence $ \underline{v}$ by $ v_1 = 1$ and for $ k \ge 1 $ recursively by: 
\begin{equation} \label{def vk}
  v_{k+1}:= 
\begin{dcases} 
    \hspace{.3cm} v_k & \text{ if } \quad 2  v_k > u_{k+1}    \\
    \left\lfloor  \frac{u_{k+1}}{ v_k} \right\rfloor \cdot v_k             & \text{otherwise}. 
\end{dcases}
\end{equation}
Obviously $v_k$ divides $ v_{k +1} $ for every $ k \ge 1$.
 
It remains to show the inequalities in \cref{r to prod}. 
First observe that if  $ v_{k-1}  > u_k /2 $  for some $ k \ge 2$, then $ v_k = v_{k-1} > u_{k}/2 $. Otherwise, $v_{k-1} \le u_k/2$ and consequently $v_k$ is equal to:
\[  
 \left\lfloor  \frac{u_{k}}{ v_{k-1}} \right\rfloor \cdot v_{k-1} \ge \frac{u_k}{2} \; . \]
In both cases, we obtained: 
\begin{equation} \label{in1 prod}
\frac{u_k}{2} \le v_k  \; . 
\end{equation}
This correspond to the left hand side inequality from \cref{r to prod} but also the fact that $  \underline{v} $  diverges to $+ \infty$. 
  
 In particular, there exists $ k _0 \in \N $ minimal such that $v_{k_0} >  1 $ and thus $ v_{k_0} = \lfloor u_{k_0} \rfloor \le u_{k_0}$. 
Assume that $ v_k \le u _k $ for $ k \ge k _0$. If  $ v_k \le  u_{k+1} / 2 $ then $v_{k+1}$ is given by: 
\[ v_{k+1}  = \left\lfloor  \frac{u_{k+1}}{ v_{k}} \right\rfloor \cdot v_{k+1}   \; ,  \]
which is at most $ u_{k+1}$.  
Otherwise  $ v_{k+1} = v_k $ which is at most $ u_k $ by the made assumption, and at most $ u_{k+1}$ as $ \underline{u}$ is non-decreasing. 
Then, the following inequality is obtained by induction: 
\begin{equation} \label{in2 prod}
 v_k \le u_k \; , 
\end{equation}
for every $ k \ge k _0$. 
Now note that \cref{in1 prod,in2 prod} conclude the proof.   
\end{proof} 
Finally, we provide: 
\begin{proof}[Proof of \cref{prop seq main}]
Let $ \underline{u}$ be the sequence provided by \cref{osc lemma} for $ \underline{a}$ and $ \underline{b}$. Let then $ \underline{v}\in R $ be provided by \cref{r to prod} for $ \underline{u}$. 
Note that:  
 \[  \frac{1}{2} a_k \le  \frac{1}{2} u_k \le v_k \le u_k \le b_k \; . \]
 Thus it holds:
 \begin{equation} \label{p1}
  \limsup_{ k \to + \infty} \frac{a_k}{v_k} \le 2 \qand \liminf_{ k \to + \infty} \frac{b_k}{v_k} \ge 1 \; .
 \end{equation}
To prove the remaining inequalities, with the notations from \cref{osc lemma}, for every $ \ell \ge 1$ the following inequalities are verified: 
 \[ v_{T_\ell} \le u_{T_\ell} = a_{T_\ell} \qand v_{T_\ell+1} \ge  \frac{u_{T_\ell+1}}{2}  = \frac{b_{T_\ell+1}}{2} \;. \] 
This implies: 
 \begin{equation} \label{p2}
 \limsup_{ k \to + \infty} \frac{a_k}{v_k} \ge \limsup_{ \ell \to + \infty} \frac{a_{T_\ell}}{v_{T_\ell}} \ge  1 \qand \liminf_{ k \to + \infty} \frac{b_k}{v_k} \le  \limsup_{ \ell \to + \infty} \frac{b_{T_\ell+1}}{v_{T_\ell+1}} \le 2   \; . 
 \end{equation}
 Then \cref{p1,p2} together imply the desired result.  
\end{proof}
\bibliographystyle{alpha}
\bibliography{arb_scales.bib} 
\end{document}